\documentclass{amsart}
\usepackage[margin=1in]{geometry}

\usepackage{color,amssymb, comment, mathrsfs,tikz,upgreek,contour, soul, colonequals, mathdots,fancyvrb}
\usepackage[colorlinks, linkcolor={blue}]{hyperref}
\usepackage{tikz-cd}
\usepackage{cleveref}
\usetikzlibrary{cd}
\usetikzlibrary{fit,shapes.geometric}
\usetikzlibrary{matrix}
\usetikzlibrary{arrows}
\usepackage{enumitem}
\usepackage{mathtools}
\usepackage{thmtools}
\usepackage{graphicx}
\usepackage[all]{xy}
\usepackage[mathlines,pagewise]{lineno}
\allowdisplaybreaks

\newtheorem{theorem}{Theorem}[section]
\newtheorem{lemma}[theorem]{Lemma}

\newtheorem{corollary}[theorem]{Corollary}
\theoremstyle{definition}
\newtheorem*{ack}{Acknowledgements}

\usepackage[only,llbracket,rrbracket,llparenthesis,rrparenthesis]{stmaryrd} 
\usepackage{accsupp}

\theoremstyle{definition}

\newtheorem*{maintheorem}{Main Results}
\newtheorem{deriveddepthlemma}[theorem]{The Derived Depth Lemma}

\theoremstyle{remark}

\newtheorem{chunk}[theorem]{}

\numberwithin{equation}{section}

\newcommand{\D}{\mathrm{D}}

\newcommand{\kk}{\Bbbk}

\newcommand{\xra}{\xrightarrow}
\newcommand{\id}{\operatorname{id}}

\newcommand{\Ext}{\operatorname{Ext}}

\newcommand{\Hom}{\operatorname{Hom}}

\newcommand{\RHom}{\operatorname{\mathbf{R}Hom}}

\newcommand{\Ass}{\operatorname{Ass}}

\newcommand{\Tor}{\operatorname{Tor}}

\newcommand{\m}{\mathfrak{m}}

\newcommand{\depth}{\operatorname{depth}}
\newcommand{\width}{\operatorname{width}}

\newcommand{\cidim}{\operatorname{CI-dim}}

\renewcommand{\mod}{\operatorname{mod}}

\newcommand{\ciuid}{\operatorname{CI^*-id}}

\newcommand{\fm}{\mathfrak{m}}
\newcommand{\Dmcpl}[1][R]{\Catsup{#1}{D}{$\fm$-com}}

\makeatletter
\DeclareFontEncoding{LS2}{}{\@noaccents}
\makeatother
\DeclareFontSubstitution{LS2}{stix}{m}{n}

\DeclareSymbolFont{largesymbolsstix}{LS2}{stixex}{m}{n}

\DeclareMathDelimiter{\lbrbrak}{\mathopen}{largesymbolsstix}{"EE}{largesymbolsstix}{"14}
\DeclareMathDelimiter{\rbrbrak}{\mathclose}{largesymbolsstix}{"EF}{largesymbolsstix}{"15}

\crefname{diagram}{diagram}{diagrams}
\crefname{diagram}{Diagram}{Diagrams}
\creflabelformat{diagram}{(#1) #2 #3}

\newcommand{\hsup}{\operatorname{hsup}}
\newcommand{\p}{\mathfrak{p}}
\newcommand{\Spec}{\operatorname{Spec}}
\newcommand{\pd}{\operatorname{pd}}
\newcommand{\hinf}{\operatorname{hinf}}
\newcommand{\qid}{\operatorname{qid}}
\newcommand{\fd}{\operatorname{fd}}

\newcommand{\qpd}[1]{\operatorname{qpd}_{#1}}
\renewcommand{\H}{\mathrm{H}}

\newcommand{\lotimes}{\otimes^{\mathbf{L}}}
\newcommand{\Supp}{\operatorname{Supp}}
\newcommand{\catb}{\sqsubset\mspace{-13mu}\sqsupset}
\newcommand{\catbb}{\sqsupset}
\newcommand{\Catsub}[3]{{\mathsf{#2}}_{#3}(#1)}
\newcommand{\Catsupsub}[4]{{\mathsf{#2}}^{\text{\upshape #3}}_{#4}(#1)}
\newcommand{\Catsup}[3]{{\mathsf{#2}}^{\text{\upshape #3}}(#1)}
\newcommand{\Db}[1][R]{\Catsub{#1}{D}{\catb}}
\newcommand{\Dfb}[1][R]{\Catsupsub{#1}{D}{f}{\catb}}
\newcommand{\Dbb}[1][R]{\Catsub{#1}{D}{\catbb}}
\newcommand{\Df}[1][R]{\Catsup{#1}{D}{f}}
\newcommand{\coker}{\operatorname{Coker}}

\newcommand{\ciid}{\operatorname{CI-id}}
\renewcommand{\Im}{\operatorname{Im}}
\newcommand{\LLam}[2][\mathfrak{a}]{\nobreak{\mathbf{L}\Lambda^{#1}#2}}
\newcommand{\cifd}{\operatorname{CI-fd}}
\newcommand{\Art}{\operatorname{Art}}

\newcommand{\n}{\mathfrak{n}}



\author[Ferraro]{Luigi Ferraro}
\address[Luigi Ferraro]{School of Mathematical and Statistical Sciences \\ University of Texas Rio Grande Valley \\ Edinburg, TX 78539, U.S.A}
\email{luigi.ferraro@utrgv.edu}
\urladdr{https://faculty.utrgv.edu/luigi.ferraro}

\author[Lyle]{Justin Lyle}
\address[Justin Lyle]{Department of Mathematics and Statistics \\ 221 Parker Hall\\
Auburn University\\
Auburn, AL 36849}
\email{jll0107@auburn.edu}
\urladdr{https://jlyle42.github.io/justinlyle/}

\begin{document}
\title{The derived depth formula for modules of finite quasi-projective dimension}
\begin{abstract}
Let $(R,\m,\kk)$ be a commutative Noetherian local ring. We prove a variety of new formulae for modules of finite quasi-projective or finite quasi-injective dimension. These include the Derived Depth Formula, itself an extension of Auslander's famous depth formula, a variation of the Derived Depth Formula for width, an extended version of Ischebeck's Formula, and a Dependency formula in the vein of Jorgensen. Several special cases of our main results are new even under stronger assumptions on the vanishing of various complete intersection dimensions.
\end{abstract}

\maketitle

\section{Introduction}\label{intro}
Let $(R,\m,\kk)$ be a commutative Noetherian local ring. The notion of quasi-projective dimension was introduced by Gheibi-Jorgensen-Takahashi in \cite{qpd} and has received intense study in the ensuing years (see \cite{BG22,qpdAb,cqpd,qpdIschebeck} for a few examples). We direct the reader to Section \ref{qpdsubsection} for the precise definition. One should take as a philosophy that a (finitely-generated) module of finite quasi-projective dimension behaves like a (finitely-generated) module over a complete intersection ring, but is orthogonal in a sense to other notions that behave analogously. For instance, invariants such as the complete intersection dimension or complexity are finite for the residue field $\kk$ if and only if $R$ is already a complete intersection (see \cite[Theorem 1.3]{AG97} and \cite[Theorem]{Gu71}, respectively), while on the other hand $\qpd{R}\kk<\infty$ for any local ring $R$ \cite[Proposition 3.6]{qpd}. However, that finiteness of $\qpd{R}M$ for every $M$ characterizes the complete intersection condition remains an open question (\cite[Question 3.9]{qpd}) in full generality, and one of some significance due to its connection to the well-studied proxy-small condition believed itself to characterize complete intersections (see \cite{BG22}). In fact, the work of \cite{BG22} uses the notion of quasi-projective dimension to establish this characterization in the case where $R$ is equipresented, or when $R$ has large cohomological support in the sense of \cite{Po21} (cf. \cite{AB00}).  

 Given the connection modules of finite quasi-projective dimension have to complete intersections, and given that \cite[Question 3.9]{qpd} remains unresolved in full generality, is natural to explore to what extent modules of finite quasi-projective dimension inherit the expected behaviors of modules over a complete intersection ring. In this work we focus primarily on three such behaviors, formulae that are known to hold, with appropriate hypotheses, over complete intersection rings. The first two of which are the so-called \emph{derived depth formula}, established for complete intersections by Christensen-Jorgensen \cite[Theorem 4.6 and Proposition 6.5]{deriveddepthwidth}, and the \emph{derived width formula} which can be readily seen to hold over complete intersections from \cite[Propositions 6.4-6.5]{deriveddepthwidth}.
 The last of which is the \emph{dependency formula}, established for modules of finite complete intersection dimension by Jorgensen \cite[Theorem 2.2]{Jorgensen}. To be precise, we present the following as our main results. In stating these results, we direct the reader to Section \ref{prelimsection} for the definitions of depth and width of complexes:

\begin{maintheorem}\label{intromaintheorem}
Suppose $M$ is a finitely generated $R$-module of finite quasi-projective dimension, and suppose $N$ is an $R$-complex with bounded homology. Then we have the following:
\begin{enumerate}
\item[$(1)$] (The Derived Depth Formula)
If $M \lotimes_R N$ has bounded homology, then
\[\depth M \lotimes_R N=\depth M+\depth N-\depth R.\]

\item[$(2)$] (The Derived Width Formula)
If $\RHom_R(M,N)$ has bounded homology, then 
\[\width \RHom_R(M,N)=\depth M+\width N-\depth R\]

\item[$(3)$] (The Dependency Formula) If $N$ and $M\lotimes_RN$ have finitely generated total homology, then
\[
\hsup(M\lotimes_RN)=\sup\{\depth R_\p-\depth_{R_\p}M_\p-\depth_{R_\p}N_\p\mid \p\in\Supp M \cap\Supp N\}.
\]

\end{enumerate}
\end{maintheorem}

(1) appears below as \Cref{thm:ddFormula}, (2) is special case of \Cref{thm:dwFormulaQPD}, and (3) appears as \Cref{thm:dependency}.

The original work of Gheibi-Jorgesen-Takahashi had previously established the ordinary depth formula, that is they establish (1) when we additionally have that $N$ is a finitely generated $R$-module satisfying $\Tor^R_{i>0}(M,N)=0$ (\cite[Theorem 4.11]{qpd}). \cite[Corollary 3.9]{JV24} relaxes this assumption on $\Tor^R_i(M,N)$ to only require $\Tor^R_{i \gg 0}(M,N)$, provided $\Tor^R_q(M,N)$ has depth at most $1$, when $q:=\sup \{i \mid \Tor^R_i(M,N) \ne 0\}$ is positive.

In addition to extending these results to derived variations, a key advantage of our approach is to relax requirements on finite generation of homology for the complex $N$, which in particular requires a bit of subtlety in managing conditions on depth, as one loses access to much of the standard theory from the finitely generated case. Relaxing finite generation requirements has the byproduct of making statements more amenable to duality, a fact we leverage to establish the derived width formula under more flexible hypotheses, some of which are new even under analogous assumptions of finite complete intersections dimensions; see \Cref{thm:dwFormulaQPD}. As a consequence, we also obtain new cases where one has the so-called Ischebeck Formula (see \cite[2.6]{ischebeck}) in \Cref{ischebeck}, directly extending \cite[Theorem 4.3]{qpdIschebeck}.

We now outline the structure of this work. In \Cref{prelimsection} we set notation and provide background needed for the body of our work. In \Cref{deriveddepthsection} we prove the Derived Depth Formula (\Cref{thm:ddFormula}) and the dependency formula (\Cref{thm:dependency}). \Cref{derivedwidthsection} proves the Derived Width Formula under a variety of hypotheses (\Cref{thm:dwFormulaQPD}) and obtains the Ischebeck Formula (\Cref{ischebeck}) as a corollary. 

\begin{ack}
The authors thank Souvik Dey and Mohsen Gheibi for helpful discussions. We also thank an anonymous referee whose careful reading and feedback improved the quality of the paper. Luigi Ferraro was partly supported by the Simons Foundation grant MPS-TSM-00007849. 
\end{ack}

\section{Preliminaries}\label{prelimsection}

Throughout, we let $(R,\m,\kk)$ be a commutative Noetherian local ring. We let $\mod(R)$ denote the category of finitely generated $R$-modules, and we let $\Art(R)$ denote the category of Artinian $R$-modules. If $M$ is an $R$-complex, we write $\sup M:=\sup \{i \mid M_i \ne 0\}$ and $\inf M:=\inf \{i \mid M_i \ne 0\}$. We also write $\hsup M:=\sup \{i \mid H_i(M) \ne 0\}$ and $\hinf M:=\inf \{i \mid H_i(M) \ne 0\}$.

We write $\D(R)$ for the derived category of $R$, $\Dbb$ for the full triangulated subcategory of $\D(R)$ whose objects have finite $\hinf$, and $\Db$ for the full triangulated subcategory of $\D(R)$ whose objects have bounded homology. We let $\Df$ denote the full triangulated subcategory of $\D(R)$ consisting of those complexes $C$ for which $H_i(C)$ is finitely generated for all $i$, and we let $\Dfb:=\Df \cap \Db$. We let $\LLam[\m]{(-)}$ denote the derived $\m$-completion functor, and we let $\Dmcpl$ denote the full subcategory of $\D(R)$ whose objects are derived $\m$-complete.

\begin{chunk}\label{defnofcbh}
If $M= \cdots \rightarrow M_{i+1} \xrightarrow{\partial^M_{i+1}} M_i \xrightarrow{\partial^M_i} M_{i-1} \rightarrow \cdots$ is a complex, we set 
\[
C_i(M)=\coker(F_{i+1}\xra{\partial_{i+1}} F_i),\quad Z_i(M)=\ker(F_i\xra{\partial_i} F_{i-1}), \quad B_i(M)=\Im(F_{i+1}\xra{\partial_{i+1}} F_i).
\]
for every $i$.

\end{chunk}

If $M,N \in \D(R)$, we set $\Ext^i_R(M,N):=H_{-i}(\RHom_R(M,N))$ and $\Tor^R_i(M,N):=H_i(M \lotimes_R N)$. 

The following definitions offer the main objects of study in this work:
\subsection*{Quasi-Projective Dimension}
A \emph{quasi-projective resolution} $F$ of a (nonzero) $R$-module $M$ is a complex of $R$-modules with the following properties:
\begin{enumerate}
\item $F_i=0$ for $i \ll 0$.
\item $F_i$ is a projective $R$-module for all $i$.
\item For all $i$, there is an $a_i \in \mathbb{N}$ for which $H_i(F) \cong M^{\oplus a_i}$, with $a_k \ne 0$ for some $k$.
\end{enumerate}
The \emph{quasi-projective dimension} $\qpd{R}M$ of $M$ is the infimum of integers $q$ for which there exists a quasi-projective resolution $F$ of $M$ with $\sup F<\infty$ and $\sup F-\hsup F=q$. The quasi-projective dimension of the zero module is set to be $-\infty$.

\subsection*{Quasi-Injective Dimension}\label{qpdsubsection}
A \emph{quasi-injective resolution} $I$ of a (nonzero) $R$-module $M$ is a complex of $R$-modules with the following properties:
\begin{enumerate}
\item $I_i=0$ for $i \gg 0$.
\item $I_i$ is an injective $R$-module for all $i$.
\item For all $i$, there is an $a_i \in \mathbb{N}$ for which $H_i(I) \cong M^{\oplus a_i}$, with $a_k \ne 0$ for some $k$.
\end{enumerate}
The \emph{quasi-injective dimension} $\qid_{R}M$ of $M$ is the infimum of integers $q$ for which there exists a quasi-injective resolution $I$ of $M$ with $\inf I>-\infty$ and $\hinf I-\inf I=q$. The quasi-injective dimension of the zero module is set to be $-\infty$.

\subsection*{Depth and Width of Complexes}
We recall the following invariants:
If $M \in \D(R)$, then the \emph{depth} of $M$ is defined as the value $\depth M:=-\hsup \RHom_R(\kk,M)$, while the \emph{width} of $M$ is defined as $\width M:=\hsup \kk \lotimes_R M$. There are many ways one may calculate $\depth$ or $\width$; see \cite[Theorems 14.4.8 and 16.2.14]{LarsBook}.

\subsection*{Complete Intersection Dimensions}
Finally, we recall the notions of various complete intersection dimensions:
A quasi-deformation of $R$ is a diagram $R \xrightarrow{\alpha} R' \xleftarrow{p} Q$
where $\alpha$ is flat and the defining ideal of the surjection $p$ is a complete intersection. We say the quasi-deformation is \emph{exceptional} if $R'$ has Gorenstein formal fibres (e.g. if $R'$ admits a dualizing complex) and if the closed fiber $R'/\m R'$ is Gorenstein.

For a nonzero $R$-complex $M$, we have the following invariants:
\begin{enumerate}
\item [$(1)$] $\cidim{M}:=\inf\{\pd_Q(M \otimes_R R')-\pd_Q(R') \mid R \xrightarrow{\alpha} R' \xleftarrow{p} Q \text{ is a quasi-deformation}\}$
\item[$(2)$] $\ciid{M}:=\inf\{\id_Q(M \otimes_R R')-\pd_Q(R')  \mid R \xrightarrow{\alpha} R' \xleftarrow{p} Q \text{ is a quasi-deformation}\}$

\item[$(3)$] $\ciuid{M}:=\inf\{\id_Q(M \otimes_R R')-\pd_Q(R')  \mid R \xrightarrow{\alpha} R' \xleftarrow{p} Q \text{ is an exceptional quasi-deformation}\}$

\item[$(4)$] $\cifd{M}:=\inf\{\fd_Q(M \otimes_R R')-\pd_Q(R')  \mid R \xrightarrow{\alpha} R' \xleftarrow{p} Q \text{ is a quasi-deformation}\}$.
\end{enumerate}
Where $\ciuid{M}$ is the \emph{upper complete intersection injective dimension} of $M$ defined in \cite[Definition 2.6]{ciInj}. We set $\cidim{M}=\ciid{M}=\ciuid{M}=\cifd{M}=0$ when $M$ is $0$.

\section{Derived Depth Formula}\label{deriveddepthsection}

The focus of this section is to prove \Cref{thm:ddFormula} and to obtain \Cref{thm:dependency} as a consequence.

We begin by recalling the following well-known result, which is a derived variation of the classical depth lemma; see \cite[Proposition 14.3.20]{LarsBook}.

\begin{deriveddepthlemma}
\label{depthlemma}
 Suppose 
\[X \to Y \to Z \to\]
is an exact triangle in $\D(R)$. Then we have the following:
\begin{enumerate}
\item $\depth Y \ge \inf\{\depth X,\depth Z\}$
\item $\depth Z \ge \inf\{\depth X-1,\depth Y\}$
\item $\depth X \ge \inf\{\depth Y,\depth Z+1\}.$
    
\end{enumerate}
\end{deriveddepthlemma}

We make use of the following lemma which allows us to control the behavior of depth of derived tensor products when passing to cokernels, cycles, and boundaries of a complex $X$:

\begin{lemma}\label{depthintriangles}
Let $X\in\Dbb[R], N\in\D(R)$ and $t\in\mathbb{Z}$. If $\depth X_i \lotimes_R N \ge t$ and $\depth \H_i(X) \lotimes_R N \ge t$ for all $i$, then we have the following for all $i$: 
\begin{enumerate}
\item[$(1)$] $\depth  C_i(X) \lotimes_R N \ge t$.
\item[$(2)$] $\depth Z_i(X) \lotimes_R N  \ge t$.
\item[$(3)$] $\depth B_i(X) \lotimes_R N  \ge t$.
\end{enumerate}
\end{lemma}

\begin{proof}
Without loss of generality, we may suppose that $\inf X=0$. For every $i$, we have short exact sequences:
\begin{equation}\label{sesdeptheq:1}
0\rightarrow H_i(X)\rightarrow C_i(X)\rightarrow B_{i-1}(X)\rightarrow0
\end{equation}
\begin{equation}\label{sesdeptheq:2}
0\rightarrow Z_i(X)\rightarrow X_i\rightarrow B_{i-1}(X)\rightarrow0
\end{equation}
\begin{equation}\label{sesdeptheq:3}
0\rightarrow B_i(X)\rightarrow X_i\rightarrow C_i(X)\rightarrow0
\end{equation}
which induce exact triangles:
\begin{equation}\label{deptheq:1}
H_i(X) \lotimes_R N\rightarrow C_i(X) \lotimes_R N\rightarrow B_{i-1}(X) \lotimes_R N\rightarrow
\end{equation}
\begin{equation}\label{deptheq:2}
Z_i(X) \lotimes_R N\rightarrow X_i \lotimes_R N\rightarrow B_{i-1}(X) \lotimes_R N\rightarrow
\end{equation}
\begin{equation}\label{deptheq:3}
B_i(X) \lotimes_R N\rightarrow X_i \lotimes_R N\rightarrow C_i(X) \lotimes_R N\rightarrow 
\end{equation}

We now proceed by induction on $i$. As $\inf X=0$, we have $C_0(X) \lotimes_R N=H_0(X) \lotimes_R N$, therefore $\depth C_0(X)\lotimes_R N \ge t$ by assumption. Similarly, we have $Z_0(X)=X_0$, so 
\[
\depth Z_0(X) \lotimes_R N=\depth X_0 \lotimes_R N \ge t.
\]
Then applying \Cref{depthlemma} to \eqref{deptheq:3}, we see that \[
\depth B_0(X) \lotimes_R N \ge \min\{\depth X_0\lotimes_R N,\depth C_0(X)\lotimes_RN+1\} \ge t,
\]
and we have established the base case.

Now suppose the claim holds for some $i-1$. Applying \Cref{depthlemma} and the inductive hypothesis to \eqref{deptheq:1}, we have $\depth C_i(X)\lotimes_R N \ge \min\{\depth H_i(X)\lotimes_R N,\depth B_{i-1}(X)\lotimes_R N\} \ge t$. Then similarly applying \ref{depthlemma} to \eqref{deptheq:2}, we see that \[
\depth Z_i(X)\lotimes_R N \ge \min\{\depth X_i\lotimes_R N,\depth B_{i-1}(X)\lotimes_R N+1)\} \ge \min\{t,t+1\}=t,
\]
while an application of \ref{depthlemma} to \eqref{deptheq:3} yields that 
\[
\depth B_i(X)\lotimes_R N \ge \min\{\depth X_i\lotimes_R N,\depth C_i(X)\lotimes_R N+1\} \ge \min\{t,t+1\}=t,
\]
and the claim follows by induction.
\end{proof}

The next lemma will be critical in the proof of the main theorem of this section, and extends the familiar fact that, if $(R,\m,\kk)$ is local and $\partial_h:F_h \to F_{h-1}$ is a map of finitely generated free $R$-modules with $\Im(\partial_h) \subseteq \m F_{h-1}$, then $\Hom_R(\kk,\partial_h \otimes N)=0$ for any finitely generated $R$-module $N$.

\begin{lemma}\label{extmaplemma}
Let $A$ be a commutative ring (not necessarily Noetherian), let $M$ be an $A$-module (not necessarily finitely generated), let $N \in \Db[A]$, and let $I$ be a finitely generated ideal of $A$. Suppose $\partial_h:F_h \to F_{h-1}$ is a map of (not necessarily finitely generated) projective $A$-modules with $\Im(\partial_h) \subseteq IF_{h-1}$. Then for all $i$, we have $\Im(\Ext^i_A(M,\partial_h \otimes N)) \subseteq I\Ext^i_A(M,F_{h-1} \otimes_A N)$.
\end{lemma}

\begin{proof}
We first factor $\partial_h$ into the maps $F_{h} \xrightarrow{a} IF_{h-1} \xrightarrow{b} F_{h-1}$ where $a$ is the restriction and $b$ is the natural inclusion. Let $x_1,\dots,x_n$ be a generating set for $I$. Then there is a surjection $q:G:=F_{h-1}^{\oplus n} \twoheadrightarrow IF_{h-1}$ given by $\begin{pmatrix} y_1 \\ \vdots \\ y_n\end{pmatrix} \mapsto \sum^n_{i=1} x_iy_i$. Since $F_h$ is projective, we may lift to a commutative diagram 

\[\begin{tikzcd}
	{F_h} & \\
	G & {IF_{h-1}}
	\arrow["{q'}"', dashed, from=1-1, to=2-1]
	\arrow["{a}", from=1-1, to=2-2]
	\arrow["q"', from=2-1, to=2-2]
\end{tikzcd}\]
Since the functor $\Ext^i_A(M,- \otimes_A N)$ preserves multiplication maps, it follows that the composition
\vspace{1mm}

\[\begin{tikzcd}[cramped,column sep=2.3em, every node/.append style={font=\small}]
	{\Ext^i_A(M,F_{h-1} \otimes_R N)^{\oplus n}} & {\Ext^i_A(M,G \otimes_A N)} & {\Ext^i_A(M,IF_{h-1} \otimes_A N)} & {\Ext^i_A(M,F_{h-1} \otimes_A N)}
	\arrow["{\cong}" {yshift=3pt}, from=1-1, to=1-2]
	\arrow["{\Ext^h_A(M,q \otimes N)}" {yshift=3pt}, from=1-2, to=1-3]
	\arrow["{\Ext^i_A(M,b \otimes N)}" {yshift=3pt}, from=1-3, to=1-4]
\end{tikzcd}\]

is given by $\begin{pmatrix} w_1 \\ \vdots \\ w_n\end{pmatrix} \mapsto \sum_{i=1}^n x_iw_i \in I\Ext^i_A(M,F_{h-1} \otimes_A N)$. On the other hand, the composition
\vspace{1mm}
\[\begin{tikzcd}[cramped,column sep=3em, every node/.append style={font=\small}]
	{\Ext^i_A(M,F_{h} \otimes_A N)} & {\Ext^i_A(M,G \otimes_A N) } & {\Ext^i_A(M,IF_{h-1} \otimes_A N)} & {\Ext^i_A(M,F_{h-1} \otimes_A N)}
	\arrow["{\Ext^i_A(M,q' \otimes_A N)}" {yshift=3pt}, from=1-1, to=1-2]
	\arrow["{\Ext^h_A(M,q \otimes N)}" {yshift=3pt}, from=1-2, to=1-3]
	\arrow["{\Ext^i_A(M,b \otimes N)}" {yshift=3pt}, from=1-3, to=1-4]
\end{tikzcd}\]

is $\Ext^i_A(M,\partial_h \otimes_A N)$. Combined with the above, it follows that 
\[\Im(\Ext^i_A(M,\partial_h \otimes N)) \subseteq I\Ext^i_A(M,F_{h-1} \otimes_A N),\] as claimed.
\end{proof}

The following immediate Corollary of \Cref{extmaplemma} covers the primary case of interest for our proof of the Derived Depth Formula:

\begin{corollary}\label{extmapzero}
Suppose $(R,\m,\kk)$ is a local ring and let $N \in \Db[R]$. If $\partial_h:F_h \to F_{h-1}$ is a map of free $R$-modules with $\Im(\partial_h) \subseteq \m F_{h-1}$, then $\Ext^i_R(\kk,\partial_h \otimes_R N)$ is the zero map for all $i$.
\end{corollary}

\begin{proof}
By \Cref{extmaplemma}, we have $\Im(\Ext^i_R(\kk,\partial_h \otimes_R N)) \subseteq \m \Ext^i_R(\kk,F_{h-1} \otimes_R N)=0$.
\end{proof}

We are now ready for the proof of the Derived Depth Formula, which is the main theorem of this section.

\begin{theorem}\label{thm:ddFormula}
Let $(R,\fm,\kk)$ be a Noetherian local ring. Let $M$ be a finitely generated $R$-module of finite quasi-projective dimension. Let $N$ be an $R$-complex such that $N,M\lotimes_RN\in\Db[R]$, then
\[
\depth M\lotimes_RN=\depth M+\depth N-\depth R.
\]
\end{theorem}

\begin{proof}
Let $F$ be a minimal quasi-projective resolution of $M$, which exists by \cite[Proposition 4.1]{qpd}. Shifting $F$, we may suppose without loss of generality that $\inf F=\hinf F=0$. We set $s=\sup F$ and $h=\hsup F$, and therefore $\qpd{R}M=s-h$. By definition, there are nonnegative integers $a_i$ such that $\H_i(F)\cong M^{\oplus a_i}$ with $a_0,a_h\neq0$. For every $i$ set $C_i:=C_i(F)$, $Z_i:=Z_i(F)$, and $B_i:=B_i(F)$.

Then the following is a minimal free resolution of $C_h$
\[
0\rightarrow F_s\rightarrow\cdots\rightarrow F_h\rightarrow0.
\]
Applying the Auslander-Buchsbaum formula to $C_h$ gives
\[
s-h=\pd_RC_h=\depth R-\depth C_h,
\]
while applying \cite[Theorem 4.4]{qpd} to $M$ gives
\[
s-h=\qpd{R}M=\depth R-\depth M,
\]
yielding 
\begin{equation}\label{eq:depth=}
\depth M=\depth C_h.
\end{equation}

For every $i$ there are exact sequences
\begin{equation}\label{eq:1}
0\rightarrow M^{\oplus a_i}\rightarrow C_i\rightarrow B_{i-1}\rightarrow0
\end{equation}
\begin{equation}\label{eq:2}
0\rightarrow Z_i\rightarrow F_i\xra{\rho_i} B_{i-1}\rightarrow0
\end{equation}
\begin{equation}\label{eq:3}
0\rightarrow B_i\xra{j_i} F_i\rightarrow C_i\rightarrow0
\end{equation}
\begin{equation}\label{eq:4}
0\rightarrow B_i\rightarrow Z_i\rightarrow M^{\oplus a_i}\rightarrow0.
\end{equation}
These exact sequences induce exact triangles
\begin{equation}\label{eq:T1}
M^{\oplus a_i}\lotimes_RN\rightarrow C_i\lotimes_RN\rightarrow B_{i-1}\lotimes_RN\rightarrow
\end{equation}
\begin{equation}\label{eq:T2}
Z_i\lotimes_RN\rightarrow F_i\lotimes_RN\xra{\rho_i\otimes N} B_{i-1}\lotimes_RN\rightarrow
\end{equation}
\begin{equation}\label{eq:T3}
B_i\lotimes_RN\xra{j_i\otimes N} F_i\lotimes_RN\rightarrow C_i\lotimes_RN\rightarrow
\end{equation}
\begin{equation}\label{eq:T4}
B_i\lotimes_RN\rightarrow Z_i\lotimes_RN\rightarrow M^{\oplus a_i}\lotimes_RN\rightarrow.
\end{equation}
We claim that $Z_i\lotimes_RN,B_i\lotimes_RN, C_i\lotimes_RN\in\Db[R]$. We prove this claim by induction on $i$. For $i=0$ one has
\[
C_0\lotimes_RN=\H_0(F)\lotimes_RN\cong M^{\oplus a_0}\lotimes_RN\cong(M\lotimes_RN)^{\oplus a_0}\in\Db[R].
\]
Since $Z_0=F_0$, that $Z_0 \lotimes_R N$ has bounded homology follows from the hypothesis $N\in\Db[R]$. The exact triangle \eqref{eq:T3} for $i=0$ shows that $B_0\lotimes_RN\in\Db[R]$. Assume the claim is true for $i-1$. Then $C_i\lotimes_RN\in\Db[R]$ by \eqref{eq:T1}, $Z_i\lotimes_RN\in\Db[R]$ by \eqref{eq:T2}, and $B_i\lotimes_RN\in\Db[R]$ by \eqref{eq:T3}.

We divide the remainder of this proof into two cases.

\textbf{Case 1:} $\depth M\lotimes_RN \le \depth N$.

Since $\pd_R C_h<\infty$ one has by \cite[Theorem 16.3.1]{LarsBook}
\[
\depth C_h\lotimes_RN=\depth C_h+\depth N-\depth R=\depth M+\depth N-\depth R,
\]
where the last equality follows from \eqref{eq:depth=}. Therefore it suffices to show that $\depth C_h\lotimes_RN=\depth M\lotimes_RN$. Since $\depth M \lotimes_R N\leq\depth N$, we may apply \Cref{depthintriangles} with $t=\depth M\lotimes_RN$ to conclude that $\depth C_h \lotimes_R N \ge \depth M \lotimes_R N$ and $\depth B_{h-1} \lotimes_R N \ge \depth M \lotimes_R N$. Assume that $\depth C_h\lotimes_RN>\depth M\lotimes_RN$. Then an application of the Derived Depth Lemma to \eqref{eq:T1} yields
\begin{align*}
\depth M\lotimes_RN&\geq\min\{\depth C_h\lotimes_RN,\depth B_{h-1}\lotimes_RN+1\}\\
&\geq\min\{\depth M\lotimes_RN+1,\depth M\lotimes_RN+1\}\\
&=\depth M\lotimes_RN+1,
\end{align*}
This is clearly not possible, proving the Derived Depth Formula in Case 1.

\textbf{Case 2:} $\depth M\lotimes_RN\geq\depth N$.

Applying \Cref{depthintriangles} with $t=\depth N$, we have $\depth B_i \lotimes_R N \ge \depth N$, $\depth Z_i \lotimes_R N \ge \depth N$ and $\depth C_i\lotimes_RN\ge\depth N$ for all $i$. Therefore
\[
\depth M+\depth N-\depth R=\depth C_h\lotimes_RN\geq\depth N,
\]
which implies $\depth M\geq\depth R$, and so they are equal since $\qpd{R}M=\depth R-\depth M$ which is nonnegative. This shows that $\qpd{R}M=0$ and therefore $C_h=F_h$. Since $\depth C_{h-1}\lotimes_RN\geq\depth N$ one has
\[
\Ext_R^{\depth N-1}(\kk, C_{h-1}\lotimes_RN)=0.
\]
Hence, the long exact sequence in Ext induced by \eqref{eq:T3} yields that the  map $\Ext_R^{\depth N}(\kk, j_{h-1}\otimes N)$ is injective. But $\partial_h\otimes N=(j_{h-1}\otimes N)\circ(\rho_h\otimes N)$, so by \Cref{extmapzero} one has the first equality below
\[
0=\Ext_R^{\depth N}(\kk,\partial_h\otimes N)=\Ext_R^{\depth N}(\kk, j_{h-1}\otimes N)\circ\Ext_R^{\depth N}(\kk,\rho_h\otimes N),
\]
and since $\Ext_R^{\depth N}(\kk, j_{h-1}\otimes N)$ is injective we deduce that $\Ext_R^{\depth N}(\kk,\rho_h\otimes N)=0$. Since $C_h=F_h$, the exact triangles \eqref{eq:T1} and \eqref{eq:T2} are the same for $i=h$. The long exact sequence in Ext yields
\[
0\rightarrow\Ext_R^{\depth N}(\kk, M^{\oplus a_h}\otimes_R N)\rightarrow\Ext_R^{\depth N}(\kk, F_h\otimes_R N)\xra{0}\Ext_R^{\depth N}(\kk, B_{i-1}\lotimes_RN),
\]
where the $0$ on the left follows by \Cref{depthintriangles}. Therefore
\[
\Ext_R^{\depth N}(\kk, M\otimes_R N)^{\oplus a_h}\cong \Ext_R^{\depth N}(\kk, F_h \otimes_R N)\neq0
\]
which shows the reverse inequality $\depth M\lotimes_RN\leq\depth N$.
\end{proof}

As a consequence of \Cref{thm:ddFormula}, we obtain the following extension of Jorgensen's Dependency Formula, see \cite[Theorem 2.2]{Jorgensen}, \cite[Theorem 3.5]{ciFlat}, and \cite[Theorem 3.6]{cqpd}. For this result, we recall that for $M \in \D(R)$, $\Supp M:=\{\p \in \Spec R \mid M_{\p} \text{ is not acyclic}\}$.

\begin{theorem}\label{thm:dependency}
Let $R$ be a Noetherian ring. Let $M$ be a finitely generated $R$-module of finite quasi-projective dimension. Let $N$ be an $R$-complex such that $N,M\lotimes_RN\in\Db[R]$, then
\[
\hsup M\lotimes_RN=\sup\{\depth R_\p-\depth_{R_\p}M_\p-\depth_{R_\p}N_\p\mid \p\in\Supp M\;\cap\Supp N\}.
\]
\end{theorem}

\begin{proof}
Let $\p\in\Supp M\cap\;\Supp N$. By \cite[Proposition 3.5(1)]{qpd} $\qpd{R_\p}M_\p<\infty$. Since $\hsup M_\p\lotimes_{R_\p}N_\p<\infty$, there is the following chain of (in)equalities
\begin{align*}
\depth R_\p-\depth_{R_\p}M_\p-\depth_{R_\p}N_\p&=-\depth_{R_\p} M_\p\lotimes_{R_\p}N_\p\\
&\leq\hsup M_\p\lotimes_{R_\p}N_\p\\
&\leq\hsup M\lotimes_RN<\infty,
\end{align*}
where the equality follows from \Cref{thm:ddFormula} and the first inequality from \cite[Proposition 16.2.16]{LarsBook}. Moreover, if $s=\hsup M\lotimes_RN $ and if $\p\in\Ass(\H_s(M\lotimes_RN))$, then the first inequality is an equality by \cite[Proposition 16.2.16(b)]{LarsBook}, while the second inequality is obviously an equality in this case.
\end{proof}

\section{Derived Width Formula}\label{derivedwidthsection}

In this section, we prove the Derived Width Formula (\Cref{thm:dwFormulaQID}), which is a variation on \cite[Propositions 6.4-6.5]{deriveddepthwidth}. While it can be readily seen from their work that the Derived Width Formula holds over complete intersections, our approach also gives a new proof of this fact (see \Cref{thm:dwFormulaQID} (2)).

While simple, the next lemma describes how width behaves under derived base change, and will ultimately be critical in establishing the derived width formula with finiteness assumptions of various complete intersection dimensions.

\begin{lemma}\label{lem:widthExt} Let $f:(R,\m,\kk)\rightarrow(S,\n,\ell)$ be a local map. Let $M\in\D(R)$, then
\[
\width_S S\lotimes_RM=\width_R M.
\]

\end{lemma}

\begin{proof}
Since $f$ is local, we have $\m \ell=0$, so there is an isomorphism of $R$-modules $\ell \cong \bigoplus_{\alpha} k$. We then have
\[\width_S S \lotimes_R M=\hinf \ell \lotimes_S (S \lotimes_R M)=\hinf \ell \lotimes_R M\]
\[=\hinf \bigoplus_{\alpha} k \lotimes_R M=\hinf k \lotimes_R M=\width_R M.\qedhere\]
\end{proof}

We are now ready for our proof of the Derived Width Formula:

\begin{theorem}\label{thm:dwFormulaQID}\label{thm:dwFormulaQPD} 
Let $(R,\m,\kk)$ be a Noetherian local ring. Let $M\in\Dfb$ and $N\in\Db[R]$ be such that $\RHom_R(M,N)\in\Db[R]$. Then
\[
\width \RHom_R(M,N)=\depth M+\width N-\depth R
\]
holds under any of the following hypotheses
\begin{enumerate}
\item $M\in\mod(R)$ and $\qpd{R}M<\infty$.
\item $\cidim_RM<\infty$.
\item $N\in\Art(R)$ and $\qid_RN<\infty$.
\item $N\in\Art(R)$ and $\ciid_RN<\infty$.
\item $N\in\Dfb$ and $\ciuid_RN<\infty$.
\end{enumerate}
\end{theorem}

\begin{proof}
\begin{enumerate}
\item Let $E$ be the injective hull of $\kk$. Let $X$ be an $R$-complex. The equalities below are proved in \cite[Proposition 16.2.22]{LarsBook}
\begin{align*}
\depth X&=\width\RHom_R(X,E)\\
\width X&=\depth\RHom_R(X,E).
\end{align*}
Therefore the Derived Width Formula for modules of finite quasi-projective dimension follows from \Cref{thm:ddFormula}
\begin{align*}
\width\RHom_R(M,N)&=\depth\RHom_R(\RHom_R(M,N),E)\\
&=\depth M\lotimes_R\RHom_R(N,E)\\
&=\depth M+\depth\RHom_R(N,E)-\depth R\\
&=\depth M+\width N-\depth R,
\end{align*}
where the second equality follows from \cite[Corollary 12.3.26(b)]{LarsBook}.

\item Same proof as in (1), where the third equality now follows from \cite[Theorem 3.3]{ciFlat}.

\item Let $E$ be the injective hull of $\kk$. The proof follows the same strategy as the proof of (1) by noticing that $\RHom_R(N,E)$ has finite quasi-projecive dimension. Indeed, let $I$ be a bounded quasi-injective resolution of $N$, then $\Hom_R(I,E)$ is a bounded complex of flat modules whose homologies are direct sums of copies of $\Hom_R(N,E)$, in particular $\Hom_R(I,E)\in\Dfb$. By \cite[Theorem 15.4.19]{LarsBook} $\Hom_R(I,E)$ has finite projective dimension. Let $P$ be a bounded semi-projective resolution of $\Hom_R(I,E)$. Then $P$ is a bounded quasi-projective resolution of $\RHom_R(N,E)$.

\item Fix a quasi-deformation $R\rightarrow R'\leftarrow S$ such that $\id_S(N\lotimes_R R')<\infty$ and $R'/\m R'$ is Artinian, which we can assume by the proof of \cite[Proposition 2.11]{ciInj}. If $X$ is an $R$-complex, we denote by $X'$ the complex $X\lotimes_RR'$. Since $N\in\Art(R)$ and $R'/\m R'$ is Artinian, it follows that $N'\in\Art(R')$.  Note that $\qid_{R'}N'<\infty$ by \cite[Proposition 2.8(2)]{qid}. The assertion now follows from the following string of equalities
\begin{align*}
\width_R\RHom_R(M,N)&=\width_{R'}\RHom_{R}(M,N)\lotimes_RR'\\
&=\width_{R'}\RHom_{R'}(M',N')\\
&=\depth_{R'}M'+\width_{R'}N'-\depth R'\\
&=\depth_RM+\width_RN-\depth R,
\end{align*}
where the first equality follows from \Cref{lem:widthExt}, the second from \cite[Proposition 12.3.31]{LarsBook}, the third from part (2), the fourth from \cite[Theorem 16.4.36]{LarsBook} and \Cref{lem:widthExt}.
\item We first show that we can reduce to the case that $R$ is complete. Indeed
\begin{align*}
\width_R\RHom_R(M,N)&=\width_{\widehat{R}}\RHom_R(M,N)\lotimes_R\widehat{R}\\
&=\width_{\widehat{R}}\RHom_{\widehat{R}}(\widehat{M},\widehat{N})
\end{align*}
where the first equality follows from \cite[Theorem 18.3.6]{LarsBook} and the second from \cite[Proposition 12.3.31]{LarsBook}. Moreover $\depth_RM=\depth_{\widehat{R}}\widehat{M}$ and $\width_RN=\width_{\widehat{R}}\widehat{N}$ by \cite[Theorem 18.3.6]{LarsBook}. Since by \cite[Corollary 3.7(a)]{ciInj} $\ciuid_{\widehat{R}}\widehat{N}<\infty$ we can assume that $R$ is complete. Let $D$ be a normalized dualizing complex. The assertion follows from the following equalities
\begin{align*}
\width\RHom_R(M,N)&=\depth\RHom_R(\RHom_R(M,N),D)\\
&=\depth M\lotimes_R\RHom_R(N,D)\\
&=\depth M+\depth\RHom_R(N,D)-\depth R\\
&=\depth M+\width N-\depth R,
\end{align*}
where the first equality follows from \cite[Proposition 18.2.3]{LarsBook}, the second one follows form \cite[Corollary 12.3.26]{LarsBook}, the third follows from \cite[Corollary 4.6(b)]{ciInj} and \cite[Theorem 3.3]{ciFlat}, and the last equality is an additional application of \cite[Proposition 18.2.3]{LarsBook}.\qedhere
\end{enumerate}
\end{proof}

Now we leverage our previous work to prove the Ischebeck Formula, see \cite[2.6]{ischebeck}, under a variety of hypotheses. We note that \Cref{ischebeck}(1) below was independently proven, under the assumptions that $N$ is a finitely generated $R$-module and $M$ is a finitely generated $R$-module of finite quasi-projective dimension in \cite[Theorem 3.2]{qpdIschebeck} and \cite[Theorems 6.13]{cqpd}. 

\Cref{ischebeck}(2) should be compared to \cite[Theorem 4.3]{qpdIschebeck} and \cite[Theorem 7.7]{cqpd} where the Ischebeck Formula was shown under the hypotheses that $M$ and $N$ are finitely generated modules and $N$ has finite quasi-injective dimension. 

\Cref{ischebeck}(5) should be compared to \cite[Theorem 4.2]{ArayaYoshino}, where Araya-Yoshino show that the Ischebeck Formula holds when $M$ and $N$ are finitely generated modules and $M$ has finite complete intersection dimension. This was later extended to the case where $M$ is a complex of finite $\cifd$ and $N$ is a complex with finitely generated homology by Sahandi-Sharif-Yassemi \cite[Theorem 3.8]{ciFlat}. 

\Cref{ischebeck}(4) and (6) should be compared to \cite[Theorem 6.1]{deriveddepthwidth}, where Christensen-Jorgensen prove the Ischebeck Formula when $M$ and $N$ are finitely generated modules, $N$ has finite Gorenstein injective dimension and the Tate cohomology modules vanish. These hypotheses hold true over rings with a dualizing complex if $N$ has finite $\ciid$ or finite $\ciuid$ by means of \cite[Theorem 5.2(c)]{ciInj}.

\begin{corollary}\label{ischebeck}
Let $(R,\fm,\kk)$ be a Noetherian local ring. Let $M\in\Dfb,N\in\Db[R]$ be complexes such that $\RHom_R(M,N)\in\Db[R]$. Then
\[
-\hinf\RHom_R(M,N)=\depth R-\depth M-\hinf N
\]
holds under any of the hypotheses below.
\begin{enumerate}
\item $M\in\mod(R), \qpd{R}M<\infty, N\in\Dfb[R]$.
\item $N\in\Art(R), \qid_{R}N<\infty$.
\item $M\in\mod(R), \qpd{R}M<\infty, N\in\Dmcpl$.
\item $N\in\Art(R)$ and $\ciid_RN<\infty$.
\item $\cidim_RM<\infty, N\in\Dmcpl$.
\item $N\in\Dfb$ and $\ciuid_RN<\infty$.
\end{enumerate}
\end{corollary}

\begin{proof}
\begin{enumerate}
\item The assertion follows from the following string of equalities
\begin{align*}
-\hinf\RHom_R(M,N)&=-\width\RHom_R(M,N)&\text{by \cite[Proposition 16.2.5(a)]{LarsBook}}\\
&=\depth R-\depth M-\width N&\text{by \Cref{thm:dwFormulaQPD}(1)}\\
&=\depth R-\depth M-\hinf N&\text{by \cite[Proposition 16.2.5(a)]{LarsBook}}.
\end{align*}

\item We first point out that $N$ is derived $\fm$-complete by \cite[Corollary 13.1.33]{LarsBook} and therefore so is $\RHom_R(M,N)$ by \cite[Proposition 13.1.31]{LarsBook}. Now one just argues as in part (1) by invoking \Cref{thm:dwFormulaQPD}(3) for the second equality.

\item As in part (2), $\RHom_R(M,N)$ is derived $\fm$-complete. The proof of (1) now applies.

\item Follows as in (2) by using \Cref{thm:dwFormulaQID}(4).

\item Follows as in (2) by using \Cref{thm:dwFormulaQID}(2).

\item The proof follows as in (1) where the second equality holds by \Cref{thm:dwFormulaQPD}(5).\qedhere
\end{enumerate}
\end{proof}

\bibliographystyle{amsplain}
\bibliography{biblio}
\end{document}